\documentclass[11pt]{amsart}
\usepackage{amssymb}
\usepackage{bbm}
\usepackage{amsfonts}
\usepackage{latexsym}
\usepackage[all]{xy}
\usepackage{amscd}
\usepackage{comment}
\usepackage{fullpage}

\numberwithin{equation}{section}
\newtheorem{theorem}[equation]{Theorem}
\newtheorem{lemma}[equation]{Lemma}

\theoremstyle{definition}

\newtheorem{example}[equation]{Example}

\newtheorem{remark}[equation]{Remark}

\newcommand{\bu}{\DOT}
\newcommand{\cH}{{\mathcal{H}}}
\newcommand{\la}{\langle}
\newcommand{\ra}{\rangle}
\newcommand{\ot}{\otimes}
\newcommand{\Wedge}{\textstyle\bigwedge}
\newcommand{\lexp}[2]{{\vphantom{#2}}^{#1}{#2}}
\renewcommand{\k}{\Bbbk}
\newcommand{\LL}{{\mathcal{L}}}

\newcommand{\mH}{\mathcal{H}}

\newcommand{\ds}{\displaystyle}
\newcommand{\Z}{{\mathbb{Z}}}
\newcommand{\DOT}{\setlength{\unitlength}{1pt}\begin{picture}(2.5,2)
                      (1,1)\put(2,3.5){\circle*{2}}\end{picture}}

\DeclareMathOperator{\ddet}{det}
\DeclareMathOperator{\gr}{gr} 

\DeclareMathOperator{\Hom}{Hom}
\DeclareMathOperator{\op}{op} 
\DeclareMathOperator{\HH}{HH}
\DeclareMathOperator{\Ext}{Ext}
\DeclareMathOperator{\sgn}{sgn}
\DeclareMathOperator{\Span}{Span}

\begin{document}

\title[PBW deformations of quantum symmetric algebras]
{PBW deformations of quantum symmetric algebras\\ and their group extensions}
\author{Piyush Shroff}
\email{piyushilashroff@gmail.com}
\address{Department of Mathematics, Texas State University,
San Marcos, Texas 78666, USA}
\author{Sarah Witherspoon}
\email{sjw@math.tamu.edu}
\address{Department of Mathematics, Texas A\&M University,
College Station, Texas 77843, USA}

\date{March 3, 2014}

\thanks{The second author was supported by NSF grants DMS-1101399
and DMS-1401016.}

\begin{abstract}
We examine PBW deformations of finite group extensions of quantum symmetric algebras, 
in particular the quantum Drinfeld orbifold algebras defined by the first author. We give a homological interpretation, in terms of Gerstenhaber brackets,  
of the necessary and sufficient conditions on parameter functions to define a quantum Drinfeld orbifold algebra, thus clarifying the conditions. In case the acting group is trivial, we determine conditions under which such a PBW deformation is a generalized enveloping algebra of a color Lie algebra;  our PBW deformations include these algebras as a special case.
\end{abstract}

\maketitle

\begin{section}{Introduction}

Poincar\'e-Birkhoff-Witt 
(PBW) deformations of quantum symmetric algebras were studied by Berger \cite{Be}
and include important classes
of examples such as the generalized enveloping algebras of color
Lie algebras. PBW deformations of group extensions of  (quantum)
symmetric algebras include many other algebras such
as rational Cherednik algebras and their generalizations studied by a number of mathematicians
(see, e.g., \cite{C,D,EG,L,RS}). 
The first author \cite{S} gave necessary and sufficient conditions on parameter
functions 
to define such  PBW deformations in this general context.
In this paper we clarify these conditions by connecting them to homological
information contained in the Gerstenhaber algebra structure of Hochschild
cohomology. We show explicitly how 
color Lie algebras are related to these PBW deformations. 

We begin with the {\bf quantum symmetric algebra} 
(or {\bf skew polynomial ring}),
\[
 S_{\bf q}(V) := \k\la v_1,\ldots,v_n \mid v_iv_j = q_{ij}v_jv_i  \mbox{ for all }
      1\leq i,j\leq n \ra ,
\]
where $\k$ is a field of characteristic 0, $V$ is a finite dimensional vector space over $\k$ with basis $v_1,v_2,\ldots,v_n$ and ${\bf q}:= (q_{ij})_{1 \leq i,j \leq n}$ is a tuple of
nonzero scalars for which $q_{ii}=1$ and $q_{ji}=q_{ij}^{-1}$ for all $i,j$. Let $G$ be a finite group acting linearly on $V$ in such a way that there is an induced action on $S_{\bf q}(V)$ by algebra automorphisms. 
For example, if $G$ acts diagonally on the chosen basis of $V$, this will be the case.
There are other possible actions as well; see, for example, \cite{BB,LS} for
actions leading to interesting deformations. 
We denote the action of $G$ by left superscript, that is, 
${}^g v$ is the element of $V$ that results from the action of $g\in G$ on $v\in V$. 
We may form the corresponding skew group algebra: 
In general for any algebra $S$ with action of $G$ by automorphisms, 
the {\bf skew group algebra} $S\rtimes G$ is $S\ot_{\k} \k G$ as a left $S$-module, 
and has the following multiplicative structure. 
Write $S\rtimes G=\oplus_{g\in G} S_g$, where $S_g=S\ot_{\k} \k g $, and 
for each $s\in S$ and $g\in G$, denote by $s\# g$ the element $s\ot g$ 
in this $g$-component $S_g$.
Multiplication on $S\rtimes G$ is determined by
$$(r\# g)(s\# h):=r(\lexp{g}s)\# gh$$
for all $r, s\in S$ and $g, h\in G$.
Then $S\rtimes G$ is a graded algebra, where elements of $V$ have degree 1
and elements of $G$ have degree 0. 

 Let $\kappa: V\times V\rightarrow (\k \oplus V) \ot _{\k} \k G$ be a bilinear map for which 
$\kappa(v_i,v_j) = - q_{ij}\kappa(v_j,v_i)$ for all $1\leq i,j\leq n$. 
Let $T(V)$ denote the tensor algebra on $V$ over $\k$, 
in which we suppress tensor symbols denoting multiplication. 
Identify the target space of $\kappa$ with the subspace of $T(V)\rtimes G$
consisting of all elements of degree less than or equal to 1. 
Define  
\begin{equation}\label{Hqk}
  \mH_{\mathbf{q}, \kappa} := (T(V)\rtimes G)/(  v_iv_j - q_{ij} v_jv_i - \kappa(v_i,v_j) \mid
     1\leq i,j\leq n),
\end{equation}
a quotient of the skew group algebra 
$T(V)\rtimes G$ by the ideal generated by all elements of the form 
$v_iv_j - q_{ij} v_jv_i -\kappa(v_i,v_j)$. 
Note that $\mH_{\mathbf{q}, \kappa}$ is a filtered algebra. 
We call $\mH_{\mathbf{q}, \kappa}$ a {\bf quantum Drinfeld orbifold algebra} if it 
is a {\bf PBW deformation} of $S_{\bf q}(V)\rtimes G$, 
that is, if its associated graded algebra is isomorphic to $S_{\bf q}(V)\rtimes G $.
Equivalently, 
the set $\{v_1^{m_1}v_2^{m_2}\cdots v_n^{m_n} \# g \mid m_i \geq 0, g \in G\}$ is a $\k$-basis for $\mH_{\mathbf{q}, \kappa}$.

Quantum Drinfeld orbifold algebras include as special cases many algebras of interest,
from rational Cherednik algebras and generalizations
(see \cite{D,EG,G,L,RS}), to
generalized enveloping algebras of color Lie algebras and quantum Lie algebras
in case $G$ is the trivial group (see \cite{K,P,PV,Wo}). 
Our analysis in this paper of the necessary and sufficient conditions on the parameter
function $\kappa$ to define a quantum Drinfeld orbifold algebra applies to all of
these algebras as special cases. 

\quad

\noindent {\bf Organization.} This paper is organized as follows.

In Section 2, we first recall from 
\cite{S} the necessary and sufficient conditions (called ``PBW conditions'') 
for $\mH_{\mathbf{q}, \kappa}$ to be a quantum Drinfeld orbifold algebra,
and show that these are all the PBW deformations of $S_{\bf q}(V)\rtimes G$
in which the  action of $G$ on $V$ is preserved.
In Section~3, we give a precise relationship between color Lie algebras and quantum 
Drinfeld orbifold algebras. 
In Section~4, we  interpret the 
PBW conditions in terms of Gerstenhaber brackets on Hochschild cohomology. 

Throughout this paper, $\k$ denotes a field of characteristic 0,
and tensor products and exterior powers are taken over $\k$. 
Some results are valid more generally in other characteristics, however some
of the homological techniques of Section~\ref{sec:hom} require the characteristic
of $\k$ to be coprime to the order of $G$, so we stick with
characteristic 0 throughout for efficiency of presentation. 

\end{section}
\begin{section}{Necessary and sufficient conditions}
We decompose $\kappa$ into its constant and linear parts:
 $\ \kappa = \kappa^C + \kappa^L$ where $\kappa^C: V\times V\rightarrow \k G$ and $\kappa^L: V\times V\rightarrow V\otimes \k G$. For each $g\in G$, let $\kappa_g: V\times V\rightarrow \k \oplus V$ be the function determined by the equation 
\[
\kappa(v,w) = \sum_{g \in G} \kappa_g(v,w) \# g ,
\]
where $\kappa_g$ also decomposes into its constant and linear parts, i.e., $\kappa_g = \kappa_g^C + \kappa_g^L$ where $\kappa_g^C: V\times V\rightarrow \k$ and $\kappa_g^L: V\times V\rightarrow V$. We recall the following theorem from \cite{S} which gives  necessary and sufficient conditions for $\mH_{\mathbf{q}, \kappa}$ to be a quantum Drinfeld orbifold algebra, that is to be a PBW deformation of $S_{\bf q}(V)\rtimes G$. 

We will need some notation to state the conditions. For each $g\in G$ and
basis vector $v_j\in V$, write $g_i^j\in \k$ for the scalars given by
$$
    {}^g v_j = \sum_{i=1}^n g^j_i v_i.
$$
The {\bf quantum $(i,j,k,l)$-minor determinant} of $g$ is
$$
   \ddet_{ijkl}(g) := g^j_l g^i_k - q_{ji} g^i_lg^j_k.
$$
The following theorem is a simultaneous generalization of main results in
\cite{LS,SW}. 

\quad

\begin{theorem}\cite[Theorem 2.2]{S}
The algebra $\mH_{\mathbf{q}, \kappa}$, defined in (\ref{Hqk}), 
is a quantum Drinfeld orbifold algebra if and only if
the following conditions hold:
\begin{enumerate}
\item For all $g,h \in G$ and $1 \leq i < j \leq n$,
\[
\kappa_g^C(v_j, v_i) = \sum_{k < l} \ddet_{ijkl}(h) \kappa_{hgh^{-1}}^C(v_l,v_k)
\quad \text{ and } \quad 
\lexp{h}{\left(\kappa_g^L(v_j, v_i)\right)} = \sum_{k < l} \ddet_{ijkl}(h) \kappa_{hgh^{-1}}^L(v_l,v_k).
\]
\text{For all distinct }$i,j,k$ \text{and for all} $g\in G$,\\
\item $q_{ji}q_{ki}v_i\kappa_g^L(v_k,v_j) - \kappa_g^L(v_k,v_j)\lexp{g}v_i - q_{kj}v_j\kappa_g^L(v_k,v_i) \\ \hspace*{1cm} + q_{ji}\kappa_g^L(v_k,v_i)\lexp{g}v_j + v_k\kappa_g^L(v_j,v_i) - q_{ki}q_{kj}\kappa_g^L(v_j,v_i) \lexp{g}v_k =0$ ,\\\\
\item $\ds\sum_{h\in G}\left(q_{ij}q_{ik}\kappa_{gh^{-1}}^L(\kappa_h^L(v_j,v_k),\lexp{h}v_i) - \kappa_{gh^{-1}}^L(v_i,\kappa_h^L(v_j,v_k)) +
q_{ik}q_{jk}\kappa_{gh^{-1}}^L(\kappa_h^L(v_k,v_i),\lexp{h}v_j)\right.\\ \hspace*{1cm} \left. - q_{ij}q_{ik}\kappa_{gh^{-1}}^L(v_j,\kappa_h^L(v_k,v_i))+ \kappa_{gh^{-1}}^L(\kappa_h^L (v_i,v_j), \lexp{h}v_k) - q_{ik}q_{jk}\kappa_{gh^{-1}}^L(v_k,\kappa_h^L (v_i,v_j))\right) \\\\
=2\left(\kappa_g^C(v_j,v_k)(v_i-q_{ij}q_{ik}\lexp{g}v_i)+\kappa_g^C(v_k,v_i)(q_{ij}q_{ik}v_j-q_{ik}q_{jk}\lexp{g}v_j)+\kappa_g^C(v_i,v_j)(q_{ik}q_{jk}v_k-\lexp{g}v_k)\right)$,\\\\
\item $\ds\sum_{h\in G}\left(q_{ij}q_{ik}\kappa_{gh^{-1}}^C(\kappa_h^L(v_j,v_k),\lexp{h}v_i) - \kappa_{gh^{-1}}^C(v_i,\kappa_h^L(v_j,v_k)) +
q_{ik}q_{jk}\kappa_{gh^{-1}}^C(\kappa_h^L(v_k,v_i),\lexp{h}v_j)\right.\\ \hspace*{0.5cm} \left. - q_{ij}q_{ik}\kappa_{gh^{-1}}^C(v_j,\kappa_h^L(v_k,v_i))+ \kappa_{gh^{-1}}^C(\kappa_h^L (v_i,v_j), \lexp{h}v_k) - q_{ik}q_{jk}\kappa_{gh^{-1}}^C(v_k,\kappa_h^L (v_i,v_j))\right)=0$ .\\\\
\end{enumerate}
\end{theorem}

We note that condition (1) above is equivalent to $G$-invariance of $\kappa$,
that is, $${}^h \kappa(v,w) = \kappa ( {}^h v, {}^h w)$$ for all $h\in G$ and $v,w\in V$. 

\begin{example} 
By modifying Example 5.5 of \cite{NW}, we obtain 
 a quantum Drinfeld orbifold algebra for which some $q_{ij}\neq 1$
and $\kappa^L \not\equiv 0$: Let $G$ be a cyclic group of order 3 
generated by $g$.
Let $q$ be a primitive third root of 1 in $\mathbb C$.
Let $V= {\mathbb {C}} ^3$ with basis $v_1,v_2,v_3$ and
$q_{21}=q, \  q_{32}=q, \ q_{13}=q$. 
Take the following diagonal action of $G$ on $V$ with respect
to this basis: 
$$
    {}^g v_1 = q v_1, \ \ \ {}^g v_2 = q^2 v_2, \ \ \ {}^g v_3 = v_3.
$$
Let 
$$\kappa(v_2,v_1) = v_3 , \ \ \ \kappa(v_3,v_2)=0, \ \ \ \kappa(v_1,v_3)=0.
$$
We check the conditions of Theorem 2.1:
Condition (1) is $G$-invariance, and we may check that indeed
$
   \kappa( {}^g v_2, {}^g v_1) = q^3 \kappa(v_2,v_1) = v_3 = {}^g \kappa
   (v_2,v_1),
$
and similarly for other triples consisting of one group element
and two basis vectors. 
Condition (2) holds:
$$
    q_{13}q_{23} v_3 \kappa^L_1(v_2,v_1)  - \kappa^L_1(v_2,v_1) v_3 + 0 + 0 =
     qq^{-1} v_3 v_3 -v_3 v_3 = 0.
$$
Finally, Conditions (3) and (4) hold as all terms are equal to 0. 
The resulting quantum Drinfeld orbifold algebra is
$$
   {\mathcal{H}}_{{\bf q},\kappa} = 
   ( T(V)\rtimes G) / (v_2v_1- qv_1v_2 - v_3 , \ v_3v_2-qv_2v_3, \
   v_1v_3-qv_3v_1 ) .
$$
\end{example}

\quad

\begin{example}\label{ex:trivial}
Another example has trivial group ($G =1$): 
Let $V= \k ^3$ with basis $v_1,v_2,v_3$ and $q_{ij}=-1$ whenever $i\neq j$. 
Let 
$$
  \kappa(v_2,v_1) = v_1 , \ \ \ \kappa(v_3,v_2)=v_3 , \ \ \
    \kappa(v_1,v_3)=0 .
$$
One may check that the conditions of Theorem~2.1 hold, and consequently 
$$\cH_{{\bf q},\kappa} = T(V)/(v_2v_1+v_1v_2-v_1, \
   v_3v_2+v_2v_3 -v_3 , \ v_1v_3 +v_3v_1 ) $$
is a quantum Drinfeld orbifold algebra. 

\end{example}

\quad

Now consider any PBW deformation $U$ of $S_{\bf q}(V)\rtimes G$ in which 
the action of $G$ on $V$ is preserved, that is, the relations $gv = {}^gv g$
($g\in G$, $v\in V$) hold in $U$. 
Then $U$ is a $\Z$-filtered algebra for which $\gr U\cong S_{\bf q}(V)\rtimes G$.
We will show next that $U\cong \cH_{{\bf q},\kappa}$ for some $\kappa$.

\begin{theorem}
Let $U$ be a PBW deformation of $S_{\bf q}(V)\rtimes G$ in which
the action of $G$ on $V$ is preserved. 
Then $U\cong \cH_{{\bf q},\kappa}$, a quantum Drinfeld orbifold algebra as defined
in (\ref{Hqk}), for some $\kappa$. 
\end{theorem}
\begin{proof}
Let $F$ denote the filtration on $U$. 
Since the associated graded algebra of $U$ is $S_{\bf q}(V)\rtimes G$, we may identify
$F_1 U$ with $(V\ot \k G) \oplus \k G$.
By hypothesis, 
$v_iv_j - q_{ij} v_jv_i \in F_1 U$ for each $i,j$, and we set
this element equal to $\kappa^L(v_i,v_j) + \kappa^C(v_i,v_j)$, where
$\kappa^L(v_i,v_j)\in V\ot \k G$ and $\kappa^C(v_i,v_j)\in \k G$,
thus defining $\kappa^L$ and $\kappa^C$ on pairs of basis elements. 
Extend bilinearly to $V\times V$, and set $\kappa = \kappa^L + \kappa^C$.
By  definition, $\kappa(v_i,v_j) = -q_{ij} \kappa(v_j,v_i)$. 
We will show that $U$ is isomorphic to  $\mH_{{\bf q},\kappa}$.
Let $\sigma: T(V)\rtimes G \rightarrow U$ be the algebra homomorphism determined by 
$\sigma(v_i\# 1) = v_i$ and $\sigma(1\# g) = g$ for all $v_i,g$.
There is indeed such  a (uniquely determined) algebra homomorphism since the action of $G$ on $V$ is preserved in $U$  by hypothesis. 
By its definition, 
$\sigma$ is surjective, since $U$ is generated by the $v_i,g$. 
We will show that the kernel of $\sigma$ is precisely the ideal generated
by all elements of the form $v_iv_j - q_{ij}v_jv_i -\kappa(v_i,v_j)$.
Let $I$ be this ideal. 
Then $\cH_{{\bf q},\kappa} = T(V)\rtimes G  / I$ by definition of $\cH_{{\bf q},\kappa}$.
By the definition of $\kappa$, the kernel of $\sigma$ contains $I$, 
and so  $\sigma$ factors through $\cH_{{\bf q},\kappa}$, inducing a surjective
homomorphism $\overline{\sigma} : \cH_{{\bf q},\kappa}\rightarrow U$.
Now in each degree, $\cH_{{\bf q},\kappa}$ and $U$ have the same dimension, as each
has associated graded algebra $S_{\bf q}(V)\rtimes G$. 
This forces $\overline{\sigma}$ to be injective as well.
\end{proof}

\end{section}

\begin{section}{Color Lie algebras}

We first recall the definition of a color Lie algebra and of 
its generalized enveloping algebras. For more details, 
see, for example, Petit and Van Oystaeyen \cite{PV}. 

Let $A$ be an abelian group and let 
$\varepsilon: A\times A\rightarrow \k^{\times}$ be an antisymmetric bicharacter,
where $\k^{\times}$ is the group of units in $\k$, that is, 
\begin{equation}
\varepsilon(a, b) \varepsilon(b, a) = 1,
\end{equation}
\begin{equation}
\varepsilon(a, bc) = \varepsilon(a, b) \varepsilon(a, c),
\end{equation}
\begin{equation}
\varepsilon(ab, c) = \varepsilon(a, c) \varepsilon(b, c),
\end{equation}
for all $a, b, c \in A$.

An $(A,\varepsilon)$-{\bf color Lie algebra} is an $A$-graded vector space $\LL=\oplus_{a\in A} \LL_a$ equipped with a bilinear bracket $[-,-]$ for which 
\begin{equation}\label{eqn:gLa}
[\LL_a, \LL_b] \subseteq \LL_{ab},
\end{equation}
\begin{equation}
[x, y] = -\varepsilon(|x|, |y|)[y,x],
\end{equation}
\begin{equation}
\varepsilon(|z|, |x|)[x,[y,z]] + \varepsilon(|x|, |y|)[y,[z,x]] + 
     \varepsilon(|y|, |z|)[z,[x,y]] = 0,
\end{equation}
whenever $a, b\in A$, and $x, y, z\in \LL$ are homogeneous elements 
(any element $x\in \LL_a$ is called homogeneous of degree $a$,
and we write $|x| = a$).

Now let $\LL$ be a color Lie algebra and 
let $\omega: \LL\times \LL\rightarrow \k$ for which
\begin{equation}\label{eqn:omega-condition}
\varepsilon(|z|, |x|)\omega(x,[y,z]) + \varepsilon(|x|, |y|)\omega(y,[z,x]) 
     + \varepsilon(|y|, |z|)\omega (z,[x,y]) = 0
\end{equation}
whenever $x, y, z\in \LL$ are homogeneous elements.
The {\bf generalized enveloping algebra} of $\LL$ associated with $\omega$ is 
$$ U_{\omega}(\LL) := T(\LL)/(v_i v_j - \varepsilon(|v_i|, |v_j|)v_j v_i - [v_i, v_j] 
 - \omega(v_i, v_j)),$$
where $v_i, v_j\in \LL$ range over a basis of homogeneous elements.

If $\LL$ is a Lie algebra, that is if $\varepsilon$ takes only the value 1,
the generalized enveloping algebras are precisely the Sridharan enveloping algebras~\cite{Sridharan}. 
If $\omega \equiv 0$, the generalized enveloping algebras are sometimes called
universal enveloping algebras, as they have a universal property generalizing that
of a universal enveloping algebra of a Lie algebra~\cite{Sch}: 
$U_{0}(\LL)$ is universal with respect to linear maps $f: \LL\rightarrow S$ for
associative algebras $S$ that take the bracket in $\LL$ to the $\varepsilon$-commutator
$[ - , - ]_{\varepsilon}$ on the image of $\LL$ in $S$ ( $[f(x),f(y)]_{\varepsilon} :=
f(x)f(y) - \varepsilon(|x|,|y|) f(y)f(x)$ for all homogeneous $x,y\in\LL$).

\begin{example}\label{Heisenberg}
Let $U$ be the associative algebra generated by $x$, $y$, and $z$ subject
to the relations
$$
   xy + yx = z , \ \ \ xz+zx=0 , \ \ \ yz+zy=0.
$$
This is the universal enveloping algebra of a color Lie algebra analogous
to the Heisenberg Lie algebra. For the definition of the color Lie algebra itself,
one may take $A = \Z_2^3$ with $\varepsilon(a,b) = (-1)^{a_1b_1+a_2b_2+a_3b_3}$
where $a =(a_1,a_2,a_3)$, $b=(b_1,b_2,b_3)$. Take $\LL$ to be of dimension~3, with
basis $x,y,z$ of degrees $(1,1,0), (1,0,1), (0, 1,1)$, respectively. 
\end{example}

The next theorem describes a  relationship between generalized enveloping
algebras $U_{\omega}(\LL)$  and quantum Drinfeld orbifold algebras 
$\cH_{{\bf q},\kappa}$. 
It states that the generalized enveloping algebras of color Lie algebras are
precisely those quantum Drinfeld orbifold algebras with $G=1$ that satisfy two
technical conditions on the parameter function $\kappa$, as detailed in the theorem. 
We will use Theorem~2.1 to prove part (a) of the next theorem; 
alternatively Berger's quantum PBW~Theorem \cite{Be} may be used. 
Part (b) largely follows from Petit and Van Oystaeyen's work on generalized
enveloping algebras of color Lie  algebras.

We will need some notation:
Letting $\mH_{{\bf q},\kappa}$ be a quantum Drinfeld orbifold algebra in which
$G=1$, scalars 
$C_l^{i,j}\in \k$ are defined by 
$$\kappa^L(v_i,v_j) = \sum_l C_l^{i,j} v_l.$$

\begin{theorem}\label{cLaqDoa}
(a) 
Let $U = \mH_{\mathbf{q},\kappa}$ be a quantum Drinfeld orbifold algebra with $G=1$.
Assume  that  for each triple $i,j,l$ of indices ($i\neq j$), 
if $C^{i,j}_l \neq 0$, then $q_{im}q_{jm}=q_{lm}$
for all $m$, and that  
 the left and right sides of the equation in
Theorem~2.1(3) are each equal to 0 for all triples of vectors.
 Then $U\cong U_{\omega}(\LL)$,  a generalized enveloping algebra for some color Lie algebra $\LL$
and  $\omega$ satisfying~(\ref{eqn:omega-condition}).

(b) Let  $U=U_{\omega}(\LL)$ be a generalized  enveloping algebra of a color Lie algebra $\LL$.
 Then $U\cong \mH_{\mathbf{q},\kappa}$, a quantum Drinfeld orbifold algebra as
defined in (\ref{Hqk}),
for some ${\bf q},\kappa$ and $G=1$.
Moreover, for each triple of $i,j,l$ indices ($i\neq j$), if $C^{i,j}_l\neq 0$, then $q_{im}q_{jm}=q_{lm}$ for all $m$,
and the left and right sides of the equation in Theorem~2.1(3) are each equal to 0
for all triples of vectors. 
\end{theorem}

\begin{remark}
One may check that Example~\ref{ex:trivial} fails the first condition
(corresponding to nonzero scalars $C^{i,j}_l$) of Theorem~\ref{cLaqDoa}(a); it is not
a generalized enveloping algebra of a color Lie algebra since it cannot 
satisfy~(\ref{eqn:gLa}).
By contrast, Example~\ref{Heisenberg} satisfies both conditions (with $\omega\equiv 0$). 
\end{remark}

\begin{proof}
(a) Let $U = \mH_{{\bf q},\kappa}$ be a quantum Drinfeld orbifold algebra 
as defined in (\ref{Hqk}), with $G=1$,
under the stated assumptions. 
Let $A=\Z^n$, a free abelian group on a choice of
generators $a_1,\ldots,a_n$ (where $n$ is the dimension of 
the vector space $V$). Let
$$
    \varepsilon( a_i, a_j) := q_{ij}
$$
for each $i,j\in \{1,\ldots,n\}$. Then
$   \varepsilon(a_i,a_j)\varepsilon(a_j,a_i) = q_{ij} q_{ji} = 1
$
for all $i,j$, that is, (3.1) holds for the generators of $A$. 
Since $A$ is a free abelian group, we may extend $\varepsilon$
uniquely to an antisymmetric bicharacter on all of $A$
via the relations (3.2), (3.3).
Set $\LL=V$.
We will show that $\LL$ is a color Lie algebra with respect to
a quotient group of $A$. 

Let 
$$
   [v_i, v_j ] := \kappa^L(v_i, v_j).
$$
Then the condition $[x,y]=-\varepsilon(|x|,|y|) [y,x]$ holds for all 
homogeneous $x,y\in\LL$, 
as a result of the condition $\kappa(v_i,v_j) = - q_{ij}\kappa(v_j,v_i)$. 
Note that by the hypothesis on $\kappa$, 
  (3.6) holds as a consequence of Theorem 2.1(3):
The left side of (3) in Theorem 2.1 is assumed equal to 0,
and this condition 
may be rewritten (with $G=1$, $g=1$, $h=1$) 
as 
$$
\begin{aligned}
&q_{ij}q_{ik} [[v_j,v_k],v_i] - [v_i,[v_j,v_k]] + q_{ik}q_{jk}[[v_k,v_i],v_j]\\
 & - q_{ij}q_{ik} [v_j,[v_k,v_i]]+ [[v_i,v_j],v_k] 
   -q_{ik}q_{jk} [v_k,[v_i,v_j]] = 0 .
\end{aligned}
$$
We wish to rewrite half of these terms in order to compare with (3.6).
By hypothesis, if 
 $C^{i,j}_l\neq 0$ for some $i,j,l$, then
$q_{lk} = q_{ik}q_{jk}$ for all $k$, so 
\begin{eqnarray*}
  [[v_i,v_j],v_k] & = & \sum_{l=1}^n C^{i,j}_l [v_l,v_k] \\
   &=& -\sum_{l=1}^n C^{i,j}_l q_{lk} [v_k,v_l]\\
  &=& -\sum_{l=1}^n C^{i,j}_l q_{ik}q_{jk}[v_k,v_l]\\
  &=& - q_{ik}q_{jk}[v_k,\sum_{l=1}^n C^{i,j}_l v_l] \ \ 
   = \ \  - q_{ik}q_{jk} [v_k, [v_i,v_j]].
\end{eqnarray*}
Similarly we have 
$[[v_k,v_i],v_j]  =  - q_{ij}q_{kj} [v_j,[v_k,v_i]]$ and 
$[[v_j,v_k],v_i]  =  - q_{ji}q_{ki} [v_i,[v_j,v_k]]$. 
Substituting into the earlier equation, it now becomes
$$
  [v_i,[v_j,v_k]]+ q_{ij}q_{ik} [v_j,[v_k,v_i]]
   + q_{ik}q_{jk} [v_k,[v_i,v_j]] =0.
$$
Multiplying by $q_{ki}$, we obtain (3.6). 

We will need to pass next to a quotient of $A$ to obtain the required
relation between the bracket and a grading on $\LL$: 
Let 
$$N=\text{ rad}(\varepsilon) = \{a\in A\mid \varepsilon(a,b)=1\ \mbox{ for all } b\in A\}.$$ 
Let $\overline{A}=A/N$ and $\LL_{\overline{a}_i} :=\Span_{\k} \{v_i\}$ for each $i$,
where $\overline{a}_i := a_i N$, and $\LL_{\overline{a}} :=0$
for all other elements $\overline{a}$ of $\overline{A}$. 
It only remains to show that 
$$
\left[\LL_{\overline{a}}, \LL_{\overline{b}}\right]\subseteq \LL_{\overline{ab}}
$$
for all $a,b\in A$. 
By hypothesis, 
 if $C^{i,j}_l \neq 0$ 
in the expression $[v_i,v_j] = \sum_l C_l^{i,j} v_l$,
then $q_{im}q_{jm}=q_{lm}$ for all $m$.
This
implies that 
$$1= q_{im}q_{jm}q_{lm}^{-1} = 
\varepsilon(a_i,a_m)\varepsilon(a_j,a_m)\varepsilon(a_l^{-1},a_m)=
\varepsilon(a_ia_ja_l^{-1},a_m).$$ 
It follows that 
$a_ia_ja_l^{-1} \in N$, so $\overline{a_ia_j} = \overline{a_l}$. Thus 
$[\LL_{\overline{a}}, \LL_{\overline{b}}]\subseteq \LL_{\overline{ab}}$ for all $a,b\in A$,
implying that  
$\LL$ is a color Lie algebra.

Now, 
for all $i, j$, let
$$
\omega(v_i, v_j) := \kappa^C(v_i, v_j).
$$
Then 
(3.7) is a consequence of Theorem 2.1(4) by a similar computation to
that above for (3.6). 
Hence $U\cong U_{\omega}(\LL)$,  a generalized enveloping algebra of the color Lie algebra $\LL$.

(b) 
Let $U=U_{\omega}(\LL)$ be a generalized  enveloping algebra of a color Lie algebra $\LL$.
Let $V=\LL$.
Choose a basis $v_1,\ldots,v_n$ of 
$V$ consisting of homogeneous elements and 
for each $i,j$, let
$$
q_{ij}=\varepsilon(|v_i|, |v_j|).
$$
Let $G=1$. Set $\kappa^L(v_i,v_j) := [v_i,v_j]$ and $\kappa^C(v_i,v_j) := \omega (v_i,v_j)$. 
By \cite[Theorem~3.1]{PV}, the associated graded algebra of $U$ is $S_{\bf q}(V)$. 
So the conditions of Theorem~2.1 must hold, and $\cH_{{\bf q},\kappa}$ is a quantum
Drinfeld orbifold algebra.
By their definitions, $\cH_{{\bf q},\kappa} = U_{\omega}(\LL)$. 

One may check that (3.4) implies that if $C^{i,j}_l\neq 0$, then $q_{im}q_{jm}=q_{lm}$
for all $m$ (similarly to computations in the proof of part (a)). 
From this and (3.6) it now follows that the left side of the equation
in Theorem 2.1(3) is equal to 0
(similarly to computations in the proof of part (a)), and therefore the
right side is also 0. 
\end{proof}

\begin{remark}
The hypothesis on the scalars $C^{i,j}_l$ in Theorem~\ref{cLaqDoa}(a) is not 
as restrictive as it appears. 
If we assume that 
$\kappa$ is a Hochschild 2-cocycle
written in the canonical form given in \cite[Theorem 4.1]{NSW}, this
condition holds automatically as a consequence of the relations defining
the space $C_1$ there (see \cite[(12)]{NSW}). 
\end{remark}

\end{section}

\begin{section}{Homological conditions}\label{sec:hom}

We first recall the definition of Hochschild cohomology and some resolutions that we will need.
For more details, see, e.g., \cite{Gerstenhaber}.

Let $R$ be an algebra over $\k$, and let $M$ be an $R$-bimodule.
Identify $M$ with a (left) $R^e$-module, where $R^e=R\ot R^{\op}$;
here, $R^{\op}$ denotes the algebra $R$ with the opposite multiplication.
The {\bf Hochschild cohomology} of $R$ with coefficients in $M$ is
$$
  \HH^{\bu}(R,M) := \Ext^{\bu}_{R^e}(R,M),
$$
where $R$ is itself considered to be an $R^e$-module under left and right multiplication.

Let $R=S\rtimes G$, where $S$ is a $\k$-algebra with an action of a group $G$
by automorphisms.
Since the characteristic of $\k$ is 0, we have
$$\HH^{\bu}(S\rtimes G)\cong \HH^{\bu}(S, S\rtimes G)^G , $$
where the superscript $G$ denotes invariants under the induced action of $G$
(see, e.g., \cite{Stefan}). 
As a graded vector space, 
$$\HH^{\bu}(S,S\rtimes G)=\Ext_{S^e}^{\bu}(S,S\rtimes G)\cong \bigoplus_{g\in G} 
  \Ext_{S^e}^{\bu}(S,S_g) , $$
where, as before, $S_g$ denotes the $g$-component $S\ot  \k g$.

Letting $S=S_{\bf q}(V)$, each summand above 
can be explicitly determined using the following 
free $S^e$-resolution of $S=S_{\bf q}(V)$, 
called its {\bf Koszul resolution} (see \cite[Proposition 4.1(c)]{W}):
\begin{equation}\label{koszul-res}
\cdots \xrightarrow{} S^e\ot\Wedge^2(V) \xrightarrow{d_2}
S^e \otimes \Wedge ^1(V) \xrightarrow{d_1}
S^e \xrightarrow{\text{mult}} S \xrightarrow{} 0,
\end{equation}
with differentials for $1\leq p\leq n$: 
\begin{equation*}
\begin{split}
&d_p(1\ot 1 \ot v_{j_1}\wedge\cdots\wedge v_{j_p})\\
&=  \sum_{i=1}^p (-1)^{i+1} \left[ \left(\prod_{s=1}^{i} q_{j_s, j_i} \right)
    v_{j_i}\ot 1 - \left(\prod_{s=i}^p q_{j_i, j_s} \right) \ot v_{j_i} \right] \ot
    v_{j_1}\wedge \cdots \wedge \hat{v}_{j_i} \wedge\cdots \wedge v_{j_p}
\end{split}
\end{equation*}
whenever $1\leq j_1<\cdots <j_p\leq n$. Applying $\Hom_{S^e}(-,S_g)$, 
dropping the term $\Hom_{S^e}(S,S_g)$, and identifying $\Hom_{S^e}(S^e\ot \Wedge^p
(V), S_g)$ with $\Hom_{\k}(\Wedge^p(V),S_g)$,
we obtain
\begin{equation}
0 \xrightarrow{}  S_g \xrightarrow{d_1^*} 
S_g \ot \Wedge^1(V^*) \xrightarrow{d_2^*} 
S_g \ot \Wedge^2(V^*) \xrightarrow{} \cdots  , 
\end{equation}
where $V^*$ denotes the vector space dual to $V$. 
Thus the space of cochains is
$$C^{\bu}=\bigoplus_{g\in G}C_{g}^{\bu},\text{ where } C_{g}^{p}=S_g\ot \Wedge^{p}(V^*),$$
for each degree $p$ and $g\in G$.
For convenience in notation, we define
$$
   v_j\wedge v_i := - q_{ji} v_i\wedge v_j
$$
whenever $i<j$ (in contrast to the standard exterior product).

We view the function $\kappa$, in the definition (\ref{Hqk}) of quantum
Drinfeld orbifold algebra,
as an element of $C^2$ by setting $\kappa(v_i\wedge v_j) = \kappa(v_i,v_j)$
for all $i,j$. 

The {\bf bar resolution} of any $\k$-algebra $R$ is:
\begin{equation}\label{bar-res}
 \cdots \xrightarrow{\delta_3} R^{\ot 4} \xrightarrow{\delta_2}
  R^{\ot 3}\xrightarrow{\delta_1} R^e \xrightarrow{\text{mult}} R \xrightarrow{ } 0 
\end{equation}
where 
$\delta_m(r_0\ot \cdots\ot r_{m+1}) =\sum_{i=0}^m (-1)^i r_0\ot\cdots\ot
    r_ir_{i+1}\ot \cdots \ot r_{m+1}$ for all $r_0,\ldots,r_{m+1}\in R$,
and the action of $R^e$ is by multiplication on the leftmost and
rightmost factors.

From \cite{W} (see also  \cite{NSW}), 
maps $\phi_p: S^e\ot \Wedge^p(V)\rightarrow S^{\ot (p+2)}$ defining
an embedding from the Koszul resolution to the bar resolution of $S=S_{\bf q}(V)$ are given by 
\begin{equation}\label{phim}
  \phi_p(1\ot 1 \ot v_{j_1}\wedge \cdots \wedge v_{j_p})
  = \sum_{\pi\in S_p} (\sgn\pi) q_{\pi}^{j_1,\ldots,j_p} \ot v_{j_{\pi(1)}}\ot \cdots
    \ot v_{j_{\pi(p)}} \ot 1
\end{equation}
where the scalars $q_{\pi}^{j_1,\ldots,j_p}$ are determined by the equation 
$q_{\pi}^{j_1,\ldots,j_p} v_{j_{\pi(1)}}\cdots v_{j_{\pi(p)}} 
  = v_{j_1}\cdots v_{j_p}$.
We wish to use maps $\psi_p: S^{\ot (p+2)}\rightarrow S^e\ot \Wedge^p(V)$ defining a chain map 
from the bar resolution to the Koszul resolution.
For our purposes, 
we need only define these maps for particular arguments in low degrees:
Let $\psi_0$ be the identity map, and $\psi_1 (1\ot v_i\ot 1) = 1\ot 1\ot v_i$. 
One checks directly that $\psi_0\delta_1$ and $d_1\psi_1$ take the same
values on elements of the form $1\ot v_i\ot 1$. We define
$$
   \psi_1(1\ot v_iv_j\ot 1) = \frac{1}{2} (q_{ij}\ot v_i + v_i\ot 1) \ot v_j
      + \frac{1}{2} (q_{ij}v_j\ot 1 + 1\ot v_j) \ot v_i
$$
for all $i,j$. (This is a different, more symmetric, choice than that made in
\cite{NW}, and it will better suit our purposes.)
Again we may check that $\psi_0 \delta_1$ takes the same values as $d_1\psi_1$
on elements of the form $1\ot v_iv_j\ot 1$.
The map $\psi_1$ may be extended to elements of degrees higher than 2 in 
$S_{\bf q}(V) ^{\ot 3}$, but we will not need these further values, and they
will not affect our calculations in the next step.
Our choices allow us to define
$$
  \psi_2(1\ot v_i\ot v_j\ot 1) = \frac{1}{2} \ot 1 \ot v_i\wedge v_j 
$$
whenever $i\neq j$, and we may check that $\psi_1\delta_2$ and $d_2\psi_2$ take
the same values on elements of the form $1\ot v_i\ot v_j \ot 1$. 
(We may take $\psi_2(1\ot v_i\ot v_i\ot 1)=0$.)
As a consequence, 
$$ 
\psi_2(v_i\ot v_j - q_{ij}v_j\ot v_i) = v_i\wedge v_j\ \text{ for } i<j
$$
(here we have dropped extra tensor factors of 1), and thus $\psi_2\phi_2$
is the identity map on input of this form, as is $\psi_1\phi_1$ on the 
input considered above.

If $\alpha$ and $\beta$ are elements of $\Hom_{R^e}(R^{\otimes 4}, R)\cong \Hom_{\k}(R^{\otimes 2}, R)$,
for any algebra $R$, then their {\bf circle product} 
$\alpha \circ \beta \in \Hom_{\k}(R^{\otimes 3}, R)$ is defined by 
$$
\alpha \circ \beta(r_1\otimes r_2\otimes r_3) := 
  \alpha(\beta(r_1\otimes r_2)\otimes r_3) - \alpha(r_1\otimes \beta(r_2\otimes r_3))
$$
for all $r_1, r_2, r_3\in R$. The {\bf Gerstenhaber bracket} in degree 2 is then
$$
[\alpha, \beta] := \alpha \circ \beta + \beta \circ \alpha.
$$
In our setting, $R=S_{\bf q}(V) \rtimes G$, whose Hochschild cohomology we identify with
the $G$-invariant subalgebra of $\HH^{\DOT} (S, S\rtimes G)$.
We may use either the bar resolution or the Koszul resolution of $S=S_{\bf q}(V)$ to gain information
about this Hochschild cohomology. 
If $\alpha$ and $\beta$ are given as cocycles on the Koszul resolution 
(\ref{koszul-res}) instead
of on the bar resolution (\ref{bar-res}), 
we apply the chain map $\psi$ to convert $\alpha$ and $\beta$ to functions on the bar 
resolution, compute the Gerstenhaber bracket of these functions, 
and then apply $\phi$ to convert back to a function on the Koszul resolution. Thus in this case, 
$$[\alpha, \beta]:= \phi^*(\psi^*(\alpha)\circ \psi^*(\beta) + \psi^*(\beta)\circ \psi^*(\alpha)).$$
Of course, the images of $\alpha$ and $\beta$ may involve elements in $S\rtimes G \setminus S$,
in which case we employ a standard technique to manage the group elements that appear in
such a computation:

\begin{lemma}
Let $\mu: S_{\bf q}(V)\ot S_{\bf q}(V) \rightarrow S_{\bf q}(V)\rtimes G$ be a Hochschild
2-cocycle representing an element of $\HH^2(S_{\bf q}(V) , S_{\bf q}\rtimes G)$. 
Then $\mu$ may be extended to 
a Hochschild 2-cocycle for $S_{\bf q}(V)\rtimes G$ by defining
$$
  \mu( r\# g, s\# h) = \mu(r,  {}^gs ) gh
$$
for all $r,s\in S_{\bf q}(V)$ and $g,h\in G$. 
\end{lemma}

\begin{proof}
This is standard; see, e.g.\ \cite[Lemma 6.2]{SW}.
Since the characteristic of $\k$ is 0, a bimodule resolution for $S_{\bf q}(V)\rtimes G$
is given by tensoring (over $\k$) a bimodule resolution for $S_{\bf q}(V)$ with $\k G$ on one side
(say the right). The action of $G$ on the left is taken to be the
semidirect product action.
\end{proof}

We are now ready to express the PBW conditions of Theorem 2.1 in terms
of the Gerstenhaber algebra structure of Hochschild cohomology. 
The following theorem  is similar to \cite[Theorem~7.2]{SW}.
It gives  necessary and sufficient conditions for $\kappa^L$ and $\kappa^C$ to define a 
quantum Drinfeld orbifold algebra, in terms of their Gerstenhaber brackets.
\begin{theorem} The algebra $\cH_{{\bf q},\kappa}$, defined in (1.1),
is a quantum Drinfeld orbifold algebra if and only if the following conditions hold:
\begin{itemize}
\item $\kappa$ is  $G$-invariant.
\item  $\kappa^L$ is a cocycle, that is $d^{*}\kappa^L=0$.
\item  $[\kappa^L, \kappa^L] = 2d^{*}\kappa^C$ as cochains.
\item  $[\kappa^C, \kappa^L] = 0$ as cochains.
\end{itemize}
\end{theorem}
\begin{proof} 
We will show that conditions (1)--(4) of Theorem 2.1 are equivalent to the four conditions 
stated in the theorem, respectively. 

We have already discussed the equivalence of the $G$-invariance condition
with Theorem 2.1(1).

Next note that 
$d^{*}\kappa^L=\kappa^L\circ d=0$ exactly when $\kappa^L\circ d_3(v_i\wedge v_j\wedge v_k)=0$ for all $i,j,k$, where 
\begin{equation*}
\begin{split}
d_3(v_i\wedge v_j\wedge v_k)
= & (v_i\ot 1-q_{ij}q_{ik}\ot v_i)\ot v_j\wedge v_k - (q_{ij}v_j\ot 1 - q_{jk}\ot v_j)\ot v_i\wedge v_k\\ & \hspace{2cm}+(q_{ik}q_{jk}v_k\ot 1-1\ot v_k)\ot v_i\wedge v_j . 
\end{split}
\end{equation*}\\
In other words, $d^{*}\kappa^L=0$ when\\
\begin{equation*}
\begin{split}
&v_i\kappa^L(v_j,v_k) - q_{ij}q_{ik}\kappa^L(v_j,v_k)v_i - q_{ij} v_j\kappa^L(v_i,v_k) \\
&+ q_{jk}\kappa^L(v_i,v_k)v_j + q_{ik}q_{jk}v_k\kappa^L(v_i,v_j)
 - \kappa^L(v_i,v_j)v_k = 0 
\end{split}
\end{equation*}
for all $i<j<k$. 
Multiply by $q_{ji}q_{ki}q_{kj}$ to obtain
\begin{equation*}
\begin{split}
&q_{ji}q_{ki}q_{kj}v_i\kappa^L(v_j,v_k) - q_{kj}\kappa^L(v_j,v_k)v_i - q_{kj}q_{ki} v_j\kappa^L(v_i,v_k) \\ &+ q_{ji}q_{ki}\kappa^L(v_i,v_k)v_j + q_{ji}v_k\kappa^L(v_i,v_j)
 - q_{ji}q_{ki}q_{kj}\kappa^L(v_i,v_j)v_k = 0 .
\end{split}
\end{equation*}
Using the relation $\kappa^L(v_l,v_m)= - q_{lm} \kappa^L(v_m,v_l)$, we may rewrite the
equation as 
\begin{equation*}
\begin{split}
&-q_{ji}q_{ki}v_i\kappa^L(v_k,v_j) + \kappa^L(v_k,v_j)v_i - q_{kj}q_{ki} v_j\kappa^L(v_i,v_k) \\ &+ q_{ji}q_{ki}\kappa^L(v_i,v_k)v_j - v_k\kappa^L(v_j,v_i)
 + q_{ki}q_{kj}\kappa^L(v_j,v_i)v_k = 0 .
\end{split}
\end{equation*}
This is precisely  Theorem 2.1(2), once we substitute $\kappa^L( - , - ) = \sum_{g\in G}
\kappa^L_g ( - , - ) \# g$, 
move group elements to the right, and apply the relation
$\kappa^L(v_i,v_k)= - q_{ik}\kappa^L(v_k,v_i)$.

Identify $\Hom_{R^e}(R^{\ot(p+2)}, - )$ with $\Hom_{\k}(R^{\ot p}, - )$ and 
$\Hom_{S^e}(S^e\ot \Wedge^p(V), - )$ with $\Hom_{\k}(\Wedge^p(V), - )$. 
If $\alpha,\beta\in\Hom_{S^e}(S^e \ot \Wedge^2(V), S\rtimes G)^G$, then by definition, 
$$(\alpha \circ \beta)(v_i \wedge v_j \wedge v_k) = (\psi^*(\alpha)\circ \psi^*(\beta))\phi(v_i \wedge v_j \wedge v_k) . $$ 
Applying (4.4), we thus have   
\begin{eqnarray*}
&&\hspace{-1cm}(\alpha \circ \beta)(v_i \wedge v_j \wedge v_k)\\
&=& (\psi^*(\alpha)\circ \psi^*(\beta))(v_i\ot v_j\ot v_k - q_{ij}v_j\ot v_i\ot v_k - q_{jk}v_i\ot v_k\ot v_j - q_{ik}q_{ij}q_{jk}v_k\ot v_j\ot v_i \\
&&\hspace{5cm}+ q_{ij}q_{jk}v_j\ot v_k\ot v_i + q_{ik}q_{jk}v_k\ot v_i\ot v_j)\\\\
&=& \psi^{*}(\alpha)(\beta(v_i\ot v_j)\ot v_k - v_i\ot \beta(v_j\ot v_k)) - q_{ij} \psi^{*}(\alpha)(\beta(v_j\ot v_i)\ot v_k - v_j\ot \beta(v_i\ot v_k))\\
&&- q_{jk} \psi^{*}(\alpha)(\beta(v_i\ot v_k)\ot v_j - v_i\ot \beta(v_k\ot v_j)) - q_{ik}q_{ij}q_{jk} \psi^{*}(\alpha)(\beta(v_k\ot v_j)\ot v_i - v_k\ot \beta(v_j\ot v_i))\\
&&+ q_{ij}q_{ik} \psi^{*}(\alpha)(\beta(v_j\ot v_k)\ot v_i - v_j\ot \beta(v_k\ot v_i)) + q_{ik}q_{jk} \psi^{*}(\alpha)(\beta(v_k\ot v_i)\ot v_j - v_k\ot \beta(v_i\ot v_j))\\\\
&=& \psi^{*}(\alpha)(\beta(v_i\ot v_j)\ot v_k - q_{ij} \beta(v_j\ot v_i)\ot v_k) + \psi^{*}(\alpha)(q_{ij}q_{ik} \beta(v_j\ot v_k)\ot v_i - q_{ik}q_{ij}q_{jk} \beta(v_k\ot v_j)\ot v_i)\\
&&+ \psi^{*}(\alpha)(q_{ik}q_{jk} \beta(v_k\ot v_i)\ot v_j - q_{jk} \beta(v_i\ot v_k)\ot v_j) - \psi^{*}(\alpha)(v_i\ot \beta(v_j\ot v_k) - q_{jk} v_i\ot \beta(v_k\ot v_j))\\
&&\hspace{-.5cm}- \psi^{*}(\alpha)(q_{ij}q_{ik} v_j\ot \beta(v_k\ot v_i) - q_{ij} v_j\ot \beta(v_i\ot v_k)) - \psi^{*}(\alpha)(q_{ik}q_{jk} v_k\ot \beta(v_i\ot v_j) - q_{ik}q_{ij}q_{jk} v_k\ot \beta(v_j\ot v_i))\\\\
&=& \psi^*(\alpha)(\beta(v_i\wedge v_j)\ot v_k) + q_{ij}q_{ik}\psi^*(\alpha)(\beta(v_j\wedge v_k)\ot v_i) +  q_{ik}q_{jk} \psi^*(\alpha)(\beta(v_k\wedge v_i)\ot v_j)\\ 
&&- \psi^*(\alpha)(v_i\ot \beta(v_j\wedge v_k)) - q_{ij}q_{ik} \psi^*(\alpha)(v_j\ot \beta(v_k\wedge v_i)) - q_{ik}q_{jk} \psi^*(\alpha)(v_k\ot \beta(v_i\wedge v_j))\\\\
&=& \psi^*(\alpha)(\beta(v_i\wedge v_j)\ot v_k - q_{ik}q_{jk} v_k\ot \beta(v_i\wedge v_j))\\
  && + q_{ij}q_{ik}\psi^*(\alpha)(\beta(v_j\wedge v_k)\ot v_i 
- q_{ji}q_{ki}v_i\ot \beta(v_j\wedge v_k))\\
  &&  + q_{ik}q_{jk}\psi^*(\alpha)(\beta(v_k\wedge v_i)\ot v_j - q_{ij}q_{kj} v_j\ot \beta(v_k\wedge v_i)) .
\end{eqnarray*}

Now assume that $\kappa^L$ is a $G$-invariant cocycle (in $C^2$) 
representing an element of $\HH^{\bu}(A,A\rtimes G)^G$.
Replacing $\alpha$ and $\beta$ by $\kappa^L$ in the above formula,
we will get the left side of Theorem 2.1(3) for each $g \in G$:
\begin{eqnarray*}
\frac{1}{2}[\kappa^L,\kappa^L](v_i\wedge v_j\wedge v_k) & = & (\kappa^L \circ \kappa^L )(v_i\wedge v_j\wedge v_k)\\
  &= & \psi^*(\kappa^L) (\kappa^L(v_i,v_j)\ot v_k - q_{ik}q_{jk} v_k\ot \kappa^L(v_i,v_j))\\
  & & + q_{ij}q_{ik}\psi^*(\kappa^L)(\kappa^L(v_j,v_k)\ot v_i - q_{ji}q_{ki} v_i\ot \kappa^L(v_j,v_k))\\
  &&+q_{ik}q_{jk}\psi^*(\kappa^L)(\kappa^L(v_k,v_i)\ot v_j - q_{ij}q_{kj} v_j\ot \kappa^L(v_k,v_i))\\
&=& \psi^*(\kappa^L) (\sum_{h\in G} \kappa_h^L(v_i,v_j)\ot {}^hv_k\# h 
    - q_{ik} q_{jk} v_k\ot \sum_{h\in G} \kappa^L_h(v_i,v_j)\# h )\\
   &&+q_{ij}q_{ik}\psi^*(\kappa^L) (\sum_{h\in G} \kappa^L_h(v_j,v_k) \ot {}^hv_i \# h 
   - q_{ji}q_{ki} v_i\ot \sum_{h\in G} \kappa^L_h (v_j,v_k) \# h)\\
 &&+ q_{ik}q_{jk} \psi^*(\kappa^L) (\sum_{h\in G} \kappa^L_h(v_k,v_i) \ot {}^h v_j \#h 
   - q_{ij}q_{kj} v_j \ot \sum_{h\in G}\kappa^L_h(v_k,v_i)\# h)\\
& = & \frac{1}{2}\sum_{h\in G}  \big(  \kappa^L (\kappa^L_h (v_i,v_j), {}^h v_k) 
   - q_{ik}q_{jk} \kappa^L (v_k, \kappa^L_h (v_i,v_j))\\
    &&+ q_{ij}q_{ik} \kappa^L(\kappa^L_h (v_j,v_k) , {}^h v_i) 
     -  \kappa^L(v_i,\kappa^L_h(v_j,v_k))\\
  &&+ q_{ik}q_{jk}\kappa^L(\kappa^L_h (v_k,v_i), {}^h v_j) 
    - q_{ij} q_{ik} \kappa^L (v_j, \kappa^L_h(v_k,v_i))\big) \# h .
\end{eqnarray*}
Indeed, this agrees with half of the left side of Theorem 2.1(3),
after rewriting $\kappa^L$ as a sum, over $g\in G$, of $\kappa_{gh^{-1}}^L$ (for each $h$), 
and then considering separately  each  expression involving a fixed $g = (gh^{-1}) h$.

A similar calculation yields $2 d^{*}\kappa^C$ 
equal to the right side of Theorem 2.1(3).
Hence, Theorem~2.1(3) is equivalent to $[\kappa^L, \kappa^L] = 2d^{*}\kappa^C$.

By again comparing coefficients of fixed $g \in G$, we see that Theorem 2.1(4) is equivalent to $[\kappa^C, \kappa^L] = 0$.
\end{proof}

\end{section}

\quad

\quad

\end{document}